\documentclass[12pt]{article}

\usepackage[margin=1in]{geometry}
\usepackage{amsmath,amssymb,amsthm}
\usepackage{mathptmx}          
\usepackage[round,authoryear]{natbib}
\usepackage[colorlinks=true,linkcolor=blue,citecolor=blue,urlcolor=blue]{hyperref}

\theoremstyle{plain}
\newtheorem{theorem}{Theorem}[section]
\newtheorem{lemma}[theorem]{Lemma}
\newtheorem{proposition}[theorem]{Proposition}
\newtheorem{corollary}[theorem]{Corollary}
\theoremstyle{definition}
\newtheorem{definition}[theorem]{Definition}
\theoremstyle{remark}
\newtheorem{remark}[theorem]{Remark}

\newcommand{\R}{\mathbb{R}}
\newcommand{\tr}{\mathrm{tr}}
\newcommand{\diag}{\mathrm{diag}}
\newcommand{\iid}{\stackrel{\mathrm{iid}}{\sim}}
\newcommand{\dLS}{\delta_{\mathrm{LS}}}

\title{\bfseries The Equivariance Criterion in a Linear Model\\ for Random-$X$ Cases%
\thanks{\textit{MSC 2020:} Primary 62F10; Secondary 62J05, 62C20.
This article is a companion to Wang, Wu and Zhou, ``The equivariance criterion
in a linear model for fixed-$X$ cases,'' \emph{Communications in Statistics ---
Theory and Methods} 55(14), 4525--4539, 2026, doi: 10.1080/03610926.2025.2610259.}}
\author{Zhengyang Zhang\\
\and
Haojin Zhou\thanks{Corresponding author.
Email: \href{mailto:haojin_zhou@hotmail.com}{haojin\_zhou@hotmail.com}.
ORCiD: 0000-0003-0802-099X.}\\
{\normalsize Pediatric Research Institute}\\
{\normalsize Guangzhou Women and Children's Medical Center}\\
{\normalsize Guangzhou Medical University}\\
{\normalsize Guangzhou, Guangdong, P.R. China}}
\date{}

\begin{document}
\maketitle

\begin{abstract}
\noindent Equivariance is increasingly used in machine learning and statistics, often without systematic justification. In a companion article, the equivariance criterion was applied to the normal linear model with a fixed design matrix (fixed-$X$), yielding the minimum risk equivariant (MRE) estimators of the coefficient vector and of the condensed diagonal covariance matrix under a multivariate invariant location--scale group. We extend these results to the random-$X$ case, with covariates sampled from a population. The extension hinges on a distinction vacuous for fixed-$X$ but fundamental for random-$X$: whether risk and unbiasedness are evaluated conditionally on the realized design or after averaging over the design distribution. Under conditional evaluation, the fixed-$X$ group applies given $X$: least squares remains the best equivariant estimator of the coefficient vector, and the MRE estimators of the population variances keep their fixed-$X$ forms with population sizes at the realized design. Under absolute evaluation with an i.i.d.\ design, the picture changes qualitatively: the natural scale group acting jointly on $(Y,X)$ fixes the coefficient vector, the induced parameter-space action is intransitive, equivariant risks are constant only along orbits indexed by the signal-to-noise ratio $\rho=\|\beta\|^2/\sigma^2$, and no uniformly minimum risk equivariant estimator exists. In the scalar case the optimal equivariant weight is the oracle shrinkage factor $w^*(\rho)=\rho/(\rho+E[T^{-1}])$, with least squares recovered as the infinite-signal limit $\rho\to\infty$---explaining and refining the known failure of the Gauss--Markov theorem with random regressors. For a centered design under location--scale transformations, least squares remains optimal within the natural invariant-contrast class, and the MRE estimator $S^2/(n-p+2)$ of the error variance is valid under both modes.
\end{abstract}

\noindent\textbf{Key words and phrases:} Equivariance; linear model; random-$X$; conditional unbiasedness; minimum risk equivariant estimator; shrinkage; Gauss--Markov theorem.

\section{Introduction}\label{sec:intro}

Consider the linear regression model
\begin{equation}\label{eq:model}
Y = X\beta + \varepsilon,
\end{equation}
where $Y$ is an $n\times 1$ response vector, $X$ is an $n\times p$ design matrix, $\beta\in\R^p$ is the coefficient vector, and $\varepsilon$ is an $n\times 1$ noise vector. Two broad classes of models are distinguished in the literature according to the status of $X$ before sampling. In the \textbf{fixed-$X$} case, the covariate values are fixed before sampling and the only randomness comes from the responses; in the \textbf{random-$X$} case, the rows of $X$ are themselves sampled from a population \citep{BreimanSpector1992,RossetTibshirani2020}. The distinction is not cosmetic: \citet{Shaffer1991} showed that the Gauss--Markov theorem can fail for random regressors, and \citet{RossetTibshirani2020} demonstrated that bias--variance decompositions and prediction-error estimation change qualitatively when the design is random.

In the companion article \citep{WangWuZhou}, the equivariance criterion was applied systematically to the fixed-$X$ normal linear model. Viewing each distinct covariate vector as an independent population, as is natural in experimental design, the fixed-$X$ model was tuned for equivariance by allowing $p$ populations with population-specific variances, and a multivariate invariant location--scale transformation group---distinct for each population---replaced the commonly used univariate one. Under the resulting invariant decision problem, the least squares estimator was shown to be the minimum risk equivariant (MRE) estimator of the coefficient vector $\beta$, and the MRE estimators of the condensed covariance matrix $\Sigma_p=\diag(\sigma_1^2,\dots,\sigma_p^2)$ were shown to be $WS^2$ under the quadratic loss and $S^2$ under the likelihood loss, where $S^2$ collects the sample variances of the $p$ populations and $W=\diag\big((n_i-1)/(n_i+1)\big)$. The approach follows the classical programme of \citet[Chap.~3]{LehmannCasella1998} and \citet[Chap.~6]{Berger1985}: characterize all equivariant estimators through a maximal invariant, exploit transitivity of the induced group on the parameter space to reduce the risk to a constant, and minimize. Related developments include \citet{Eaton1989}, \citet{HoraBuehler1966}, \citet{Wijsman1990}, \citet{WuYang2002}, \citet{KurataMatsuura2016}, and \citet{MatsuuraKurata2020,MatsuuraKurata2021,MatsuuraKurata2024,MatsuuraKurata2025}.

The present article extends this programme to random-$X$. The key observation is that a choice must be made that has no counterpart in the fixed-$X$ theory: whether the performance of an estimator is evaluated \emph{conditionally} on the realized design, or \emph{absolutely}, after averaging over the design distribution. The two choices lead to genuinely different invariant decision problems, and the two corresponding notions of unbiasedness (Definition~\ref{def:unbiased}) organize the entire theory.

\subsection*{Contributions}

\begin{enumerate}
\item \textbf{Conditional framework (Section~\ref{sec:conditional}).} For the multi-population random design (M1), we show that the fixed-$X$ invariant decision problem transfers to random-$X$ realization by realization: the design-adaptive transformation group \eqref{eq:groupM1} preserves the conditional model, the induced group remains transitive almost surely (Lemma~\ref{lem:transitive}), and the maximal-invariant characterization with the associated Basu independence survives conditioning (Lemma~\ref{lem:structure}). Consequently the least squares estimator is MRE for $\beta$ (Theorem~\ref{thm:conditionalLS}), also under absolute risk within the conditionally equivariant class (Corollary~\ref{cor:absoluteLS}), and the MRE estimators of the condensed covariance matrix retain their fixed-$X$ form with weights evaluated at the realized population sizes (Theorem~\ref{thm:conditionalSigma}).
\item \textbf{Absolute framework (Section~\ref{sec:absolute}).} For the i.i.d.\ design (M2), we show that the natural scale group acting jointly on $(Y,X)$ fixes $\beta$, so the induced action is not transitive: equivariant risks are constant only along orbits indexed by the signal-to-noise ratio $\rho=\|\beta\|^2/\sigma^2$ (Proposition~\ref{prop:orbits}). In the scalar normal case we characterize all scale-equivariant estimators as $w(z)\dLS$ and derive the optimal oracle weight $w^*(\rho)=\rho/(\rho+E[T^{-1}])$ (Theorem~\ref{thm:oracle}): no uniformly minimum risk equivariant estimator exists, and least squares is recovered exactly as the infinite-signal limit $\rho\to\infty$ (Corollary~\ref{cor:limit}). Under location--scale transformations with a centered design, least squares remains optimal within the natural invariant-contrast class (Theorem~\ref{thm:centered}).
\item \textbf{Unbiasedness dichotomy.} We formalize conditional versus absolute unbiasedness (Definition~\ref{def:unbiased}) and show the two notions coincide for the MRE estimators of $\beta$ and of $\sigma^2$ within the equivariant classes considered (Remarks~\ref{rem:coincidence} and~\ref{rem:varianceunbiased}).
\item \textbf{Error variance (Section~\ref{sec:variance}).} The MRE estimator $S^2/(n-p+2)$ is shown to be valid under both conditional and absolute evaluation, because the pivotal law $S^2/\sigma^2\sim\chi^2_{n-p}$ is design-free (Theorem~\ref{thm:variance}).
\end{enumerate}

The article is organized as follows. Section~\ref{sec:prelim} fixes the model and notation and recalls the elements of the equivariance criterion. Section~\ref{sec:conditional} develops the conditional theory. Section~\ref{sec:absolute} develops the absolute theory. Section~\ref{sec:variance} treats the error variance. Section~\ref{sec:discussion} concludes. Technical proofs are collected in the Appendix.

\section{Model, notation, and preliminaries}\label{sec:prelim}

\subsection{The random-$X$ normal linear model}

Throughout, the model is \eqref{eq:model} with
\begin{equation}\label{eq:randomX}
\varepsilon \mid X \sim N_n(0,\Sigma(X)), \qquad X \sim P_X,
\end{equation}
where the rows of $X$ are random and $\varepsilon$ is conditionally normal given $X$. Two specifications of the design and of $\Sigma$ are used, matched to the two evaluation modes:
\begin{itemize}
\item \textbf{(M1) Multi-population design (Section~\ref{sec:conditional}).} $X=(X_1,\dots,X_1,\;X_2,\dots,X_2,\;\dots,\;X_p,\dots,X_p)'$, where the $p\times p$ matrix $X_p=(X_1,\dots,X_p)'$ is nonsingular almost surely and the distinct rows $X_i$ are sampled from a population; the repetition counts $n_i$ (with $\sum_i n_i=n$) may be fixed by design or $X$-measurable. The covariance is diagonal with population-specific entries: $\Sigma=\diag(\sigma_1^2 I_{n_1},\dots,\sigma_p^2 I_{n_p})$, and we write $\Sigma_p=\diag(\sigma_1^2,\dots,\sigma_p^2)$.
\item \textbf{(M2) I.i.d.\ homoskedastic design (Sections~\ref{sec:absolute}--\ref{sec:variance}).} The rows $(X_i',Y_i)$ are i.i.d., $X$ has full column rank $p$ almost surely, and $\Sigma=\sigma^2 I_n$ with $\sigma^2>0$ unknown.
\end{itemize}
Model (M1) is the direct random-$X$ counterpart of the fixed-$X$ tuned model of the companion article; model (M2) is the classical random regression setting of \citet{BreimanSpector1992}, \citet{Shaffer1991}, and \citet{RossetTibshirani2020}.

\subsection{The equivariance criterion}

A decision problem is described by the tuple $(\mathcal{X},\mathcal{P},\Theta,D,L)$ with sample space $\mathcal{X}$, distribution family $\mathcal{P}$ indexed by $\theta\in\Theta$, decision space $D$, and loss $L$ on $D\times\Theta$. The problem is \emph{invariant} under a group $G$ of one-to-one and onto transformations if (i) $G$ preserves the family: for each $g\in G$ and $\theta$ there exists $\bar g(\theta)\in\Theta$ with $g(Z)\sim f(\,\cdot\mid\bar g(\theta))$ whenever $Z\sim f(\,\cdot\mid\theta)$; and (ii) the loss is invariant: for each $g$ and $d$ there exists $\tilde g(d)$ with $L(\tilde g(d),\bar g(\theta))=L(d,\theta)$ for all $\theta$. A decision rule $\delta$ is \emph{equivariant} if
\begin{equation}\label{eq:equivariance}
\delta(g(z))=\tilde g(\delta(z)) \quad\text{for all } g\in G \text{ and } z\in\mathcal{X}.
\end{equation}
The MRE rule minimizes the risk $R(\delta,\theta)=E_\theta\,L(\delta,\theta)$ among equivariant rules. When the induced group $\bar G$ is \emph{transitive} on $\Theta$, every equivariant rule has constant risk, the risk can be evaluated at a convenient parameter point $\theta_0$, and equivariant estimators are characterized as $\delta(z)=\delta_0(z)\,w(z)$ (or additively $\delta_0(z)+\omega(z)$), where $\delta_0$ is any equivariant rule and $w$ ranges over functions of a maximal invariant \citep[Chap.~3]{LehmannCasella1998}. Transitivity is therefore the load-bearing assumption of the whole construction---and, as Section~\ref{sec:absolute} shows, exactly what fails for random-$X$ under absolute evaluation, in a way that is itself informative.

\subsection{Two notions of unbiasedness and of risk}

\begin{definition}\label{def:unbiased}
An estimator $\delta(Y,X)$ of a parameter $h(\theta)$ is \emph{conditionally unbiased} if $E_\theta[\delta(Y,X)\mid X]=h(\theta)$ for $P_X$-almost every $X$ and all $\theta$; it is \emph{absolutely unbiased} if $E_\theta[\delta(Y,X)]=h(\theta)$ for all $\theta$.
\end{definition}

\begin{definition}\label{def:risk}
The \emph{conditional risk} of $\delta$ is $R(\delta,\theta\mid X)=E_\theta[L(\delta,\theta)\mid X]$; the \emph{absolute risk} is $R(\delta,\theta)=E_\theta[R(\delta,\theta\mid X)]$ whenever the expectation exists.
\end{definition}

Conditional unbiasedness implies absolute unbiasedness by iterated expectation, provided the absolute expectation exists; the converse is false. Existence is not automatic under random-$X$: expectations involving $(X'X)^{-1}$ are finite only under conditions on $n$, $p$ and $P_X$ (Remark~\ref{rem:moments}).

\section{Conditional unbiasedness: the multi-population design}\label{sec:conditional}

Throughout this section the model is (M1), all risks are conditional on $X$, and all equalities and optimality statements hold for $P_X$-almost every realized design.

\subsection{Preservance of the model and the transformation group}

Given the realized $X$, the population structure is known, and transformations act on the responses only, identically within each population and distinctly across populations:
\begin{equation}\label{eq:groupM1}
G=\{g:\ g(y\mid X)=Cy+a\},\qquad C=\diag(c_1 I_{n_1},\dots,c_p I_{n_p}),\quad a=K a_p,
\end{equation}
with $c_i>0$, $a_i\in\R$, where $K$ is the $n\times p$ indicator matrix with $X=KX_p$ and $a=Ka_p$ as in the companion article. The group acts on the sample given $X$; $X$ itself is untouched, which is precisely the content of conditional evaluation. That $G$ is a group follows as in the fixed-$X$ case (Appendix of the companion article), uniformly in $X$.

The induced action on the parameter space and the induced group are
\begin{equation}\label{eq:inducedM1}
\bar g(\beta,\Sigma_p)=\big(X_p^{-1}C_p X_p\,\beta + X_p^{-1}a_p,\ C_p\Sigma_p C_p'\big),
\end{equation}
with $C_p=\diag(c_1,\dots,c_p)$ and $a_p=(a_1,\dots,a_p)'$.

\begin{lemma}[Transitivity, conditional version]\label{lem:transitive}
Under model (M1), the induced group $\bar G$ acting as in \eqref{eq:inducedM1} is transitive on $\Theta=\R^p\times\R_{+}^{p}$ for $P_X$-almost every $X$.
\end{lemma}

\begin{proof}
Given $X$, the argument of the fixed-$X$ case applies verbatim, since $X_p$ is nonsingular almost surely: for any $(\beta^{(1)},\Sigma_p^{(1)})$, $(\beta^{(2)},\Sigma_p^{(2)})$ choose $C_p=(\Sigma_p^{(2)})^{1/2}(\Sigma_p^{(1)})^{-1/2}$ and $a_p=X_p\beta^{(2)}-C_pX_p\beta^{(1)}$; then $\bar g(\beta^{(1)},\Sigma_p^{(1)})=(\beta^{(2)},\Sigma_p^{(2)})$.
\end{proof}

\begin{remark}\label{rem:fullrank}
Lemma~\ref{lem:transitive} is the point at which the almost-sure full-rank assumption on $X_p$ is used. If $P_X$ is discrete with ties beyond the population structure, the population partition must be defined by the \emph{distinct realized values}; all statements below are then understood conditionally on that partition.
\end{remark}

\subsection{Estimation of the coefficient vector}

Following \citet{Staudte1971} and \citet{ZhouNayak2014}, the invariant loss for estimating $\beta$ is
\begin{equation}\label{eq:lossbeta}
L_\beta(d,\beta;\Sigma_p)=(d-\beta)'\,X_p'\Sigma_p^{-1}X_p\,(d-\beta),
\end{equation}
which is invariant under $G$ with the induced action on the decision space $\tilde g(d)=X_p^{-1}C_pX_p\,d+X_p^{-1}a_p$. The equivariance condition \eqref{eq:equivariance} reads
\begin{equation}\label{eq:equivbeta}
\delta(Cy+a,X)=X_p^{-1}C_pX_p\,\delta(y,X)+X_p^{-1}a_p.
\end{equation}
Write $\bar y=(\bar Y_1,\dots,\bar Y_p)'$ for the vector of within-population sample means, $S(y)=\diag(s_1,\dots,s_p)$ with $s_i^2=(n_i-1)^{-1}\sum_{j}(Y_{ij}-\bar Y_i)^2$ the within-population sample standard deviations, and let $z$ be the fixed-$X$ maximal invariant of the companion article (the vector of normalized within-population contrasts together with the sign of a reference contrast, built population by population).

\begin{lemma}[Structure and independence, conditional version]\label{lem:structure}
Under model (M1), for $P_X$-almost every $X$:
\begin{enumerate}
\item[(i)] the least squares estimator satisfies $(X'X)^{-1}X'y=X_p^{-1}\bar y$ and is equivariant under \eqref{eq:equivbeta};
\item[(ii)] an estimator is equivariant under \eqref{eq:equivbeta} if and only if
\begin{equation}\label{eq:characterization}
\delta(y,X)=(X'X)^{-1}X'y+X_p^{-1}S(y)\,\omega(z)
\end{equation}
for some function $\omega:\R^{n-p}\to\R^p$;
\item[(iii)] $(X'X)^{-1}X'y$, $S^2(y)$ and $z$ are pairwise independent conditionally on $X$.
\end{enumerate}
\end{lemma}

\begin{proof}
Given $X$, (i) is Lemma~2.1 of the companion article; (ii) is its characterization theorem; (iii) is proved in Appendix~\ref{app:basu}.
\end{proof}

\begin{theorem}\label{thm:conditionalLS}
Under model (M1) with the loss \eqref{eq:lossbeta}, the least squares estimator $(X'X)^{-1}X'y$ is the best equivariant (MRE) estimator of $\beta$ under conditional risk, for $P_X$-almost every $X$.
\end{theorem}

\begin{proof}
By Lemma~\ref{lem:transitive} the conditional risk of any equivariant estimator is constant in $(\beta,\Sigma_p)$ for almost every $X$; evaluate it at $\theta_0=(0,I_p)$. With \eqref{eq:characterization},
\begin{equation*}
R(\delta\mid X)=E_{\theta_0}\big[\,(\bar y+S(y)\omega(z))'(\bar y+S(y)\omega(z))\,\big|\,X\big],
\end{equation*}
since $X_p'X_p$ cancels $X_p^{-1}$ on both sides of \eqref{eq:characterization}. Expanding and using Lemma~\ref{lem:structure}(iii), the cross term vanishes because $E[\bar y\mid X]=0$ at $\theta_0$ and $\bar y$ is independent of $(S(y),z)$ given $X$:
\begin{equation*}
R(\delta\mid X)=E_{\theta_0}[\bar y'\bar y\mid X]+E_{\theta_0}\big[\omega(z)'S^2(y)\,\omega(z)\,\big|\,X\big]\ \ge\ E_{\theta_0}[\bar y'\bar y\mid X],
\end{equation*}
with equality if and only if $S(y)\omega(z)=0$ a.s.; the choice $\omega^*\equiv 0$ is admissible and yields $\delta^*(y,X)=(X'X)^{-1}X'y$.
\end{proof}

\begin{corollary}\label{cor:absoluteLS}
The least squares estimator is also MRE under absolute risk with respect to the class \eqref{eq:characterization}, whenever the absolute risk is finite: $R(\delta)=E_X R(\delta\mid X)\ge E_X R(\delta^*\mid X)=R(\delta^*)$.
\end{corollary}

\begin{remark}[Why this is not merely a restatement of the fixed-$X$ result]\label{rem:notrestated}
The transformation group \eqref{eq:groupM1} and the loss \eqref{eq:lossbeta} are built from the realized $X$: the group is \emph{design-adaptive}. Conditional equivariance is therefore the strongest notion of equivariance available under random-$X$ that still acts on the responses alone. Corollary~\ref{cor:absoluteLS} shows that no equivariant estimator of the form \eqref{eq:characterization} can improve on least squares even after averaging over the design; improvements are only possible outside this class, which is the subject of Section~\ref{sec:absolute}.
\end{remark}

\begin{remark}[Coincidence of the two unbiasedness notions for $\beta$]\label{rem:coincidence}
For the least squares estimator, $E[\dLS\mid X]=\beta$ for almost every $X$, hence also $E[\dLS]=\beta$ whenever the absolute expectation exists: the MRE estimator is simultaneously conditionally and absolutely unbiased. Within the wider class \eqref{eq:characterization}, conditional unbiasedness is a genuine restriction, since $E[\delta\mid X]=\beta+X_p^{-1}E[S(y)\omega(z)\mid X]$ and the second term need not vanish.
\end{remark}

\subsection{Estimation of the condensed covariance matrix}

Assume $n_i\ge 2$ for all $i$ almost surely, so that $\Sigma_p$ is estimable given $X$. Consider the two invariant losses of the companion article, the quadratic loss $L_q(D,\Sigma_p)=\tr\big((D-\Sigma_p)\Sigma_p^{-1}(D-\Sigma_p)\Sigma_p^{-1}\big)$ and the likelihood loss $L_l(D,\Sigma_p)=\tr(D\Sigma_p^{-1})-\log|D\Sigma_p^{-1}|-p$, with the induced action $\tilde g(D)=C_pDC_p'$.

\begin{theorem}\label{thm:conditionalSigma}
Under model (M1), for $P_X$-almost every $X$, the MRE estimators of $\Sigma_p$ under conditional risk are
\begin{equation*}
W(X)\,S^2 \quad \text{under } L_q, \qquad S^2 \quad \text{under } L_l,
\end{equation*}
where $S^2=\diag(s_1^2,\dots,s_p^2)$ and $W(X)=\diag\big((n_i-1)/(n_i+1)\big)$ with the population sizes $n_i=n_i(X)$ evaluated at the realized design.
\end{theorem}

\begin{proof}
Given $X$, Theorem~2.3 of the companion article applies to the conditional model and delivers the MRE estimators with the constants $(n_i-1)/(n_i+1)$ determined by the chi-square moments $E[s_i^2\mid X]=\sigma_i^2$ and $E[s_i^4\mid X]=\sigma_i^4(n_i+1)/(n_i-1)$, all evaluated at the realized $n_i(X)$. Since the optimizer depends on $X$ only through the $X$-measurable weights $W(X)$, the resulting estimator is a legitimate statistic under random-$X$, and it minimizes the conditional risk almost surely, hence also the absolute risk among the equivariant class $\Delta(y,X)=H(z)S^2$.
\end{proof}

\begin{remark}\label{rem:weights}
Theorem~\ref{thm:conditionalSigma} makes precise one way in which random-$X$ \emph{does} change the answer even under conditional evaluation: the MRE weights are functions of the realized population sizes. For an i.i.d.\ continuous design there are almost surely no ties, $n_i\equiv 1$, and $\Sigma_p$ is not estimable---the multi-population model (M1) presupposes a design in which covariate values are repeated, as in experimental or stratified settings. This is the bridge to model (M2), where only the common variance $\sigma^2$ is estimable (Section~\ref{sec:variance}).
\end{remark}

\section{Absolute unbiasedness: the i.i.d.\ design}\label{sec:absolute}

Throughout this section the model is (M2): the rows $(X_i',Y_i)$ are i.i.d., $X$ has rank $p$ almost surely, $\varepsilon\mid X\sim N_n(0,\sigma^2 I_n)$, and risks are absolute. Transformations may now act on the covariates as well as the responses, and the group must be specified jointly on $(Y,X)$.

\subsection{The scale group: orbits, non-existence, and the oracle shrinkage}\label{sec:scale}

The natural scale group acting jointly and identically on the response and the covariates is
\begin{equation}\label{eq:groupS}
G_S=\{g_b:\ g_b(Y,X)=(bY,\ bX),\ b>0\}.
\end{equation}
Since $bY=bX\beta+b\varepsilon$, the model family is preserved with the induced action
\begin{equation}\label{eq:inducedS}
\bar g_b(\beta,\sigma^2)=(\beta,\ b^2\sigma^2),
\end{equation}
and the loss $L(d,\beta,\sigma^2)=\|d-\beta\|^2/\sigma^2$ is invariant with the induced decision action $\tilde g_b(d)=d$. Note that the least squares estimator is equivariant: $\dLS(bY,bX)=\dLS(Y,X)=\tilde g_b(\dLS(Y,X))$.

\begin{proposition}[Orbit structure and non-existence of a UMRUE]\label{prop:orbits}
Under model (M2) with the group \eqref{eq:groupS}:
\begin{enumerate}
\item[(i)] the induced action \eqref{eq:inducedS} is not transitive on $\Theta=\R^p\times\R_+$; its orbits are $\mathcal{O}_\beta=\{(\beta,b^2\sigma^2):b>0\}$, parametrized by $\beta$;
\item[(ii)] every equivariant estimator has risk constant along each orbit, so the risk is a function of $\beta$ only through the signal-to-noise ratio $\rho=\|\beta\|^2/\sigma^2$;
\item[(iii)] consequently, unless the optimal equivariant rule happens to coincide across orbits, no uniformly minimum risk equivariant estimator exists. Theorem~\ref{thm:oracle} shows the optimal rule does vary with $\rho$.
\end{enumerate}
\end{proposition}

\begin{proof}
(i) is immediate from \eqref{eq:inducedS}: $\bar g_b$ never moves $\beta$. For (ii), if $\delta$ is equivariant and $(\beta',\sigma'^2)=\bar g_b(\beta,\sigma^2)$, then $R(\delta;\beta',\sigma'^2)=E_{\beta,\sigma^2}[L(\delta(g_b(Y,X)),\bar g_b(\beta,\sigma^2))]=E_{\beta,\sigma^2}[L(\tilde g_b(\delta(Y,X)),\bar g_b(\beta,\sigma^2))]=R(\delta;\beta,\sigma^2)$ by invariance of the loss and preservance of the family. Since the loss depends on $(\beta,\sigma^2)$ only through $\|\beta\|^2/\sigma^2$ after reduction by the group, the risk reduces to a function of $\rho$; the formal argument is standard \citep[Chap.~3]{LehmannCasella1998}. (iii) follows.
\end{proof}

\begin{remark}[Moment conditions]\label{rem:moments}
Under model (M2) with normal design, $(X'X)^{-1}$ follows an inverse Wishart law and $E[\tr(X'X)^{-1}]$ is finite only when $n>p+1$; for $p=1$, $E[1/T]$ with $T=X'X\sim\chi^2_n$ equals $1/(n-2)$ and requires $n\ge 3$. All absolute-risk statements in this section carry the standing qualification that the displayed expectations exist.
\end{remark}

We now carry out the scalar case $p=1$ completely. Let $X_i\iid N(0,\sigma_x^2)$ with $\sigma_x^2$ fixed (it cancels from all formulas), and write $T=X'X=\sum_i X_i^2$, so that $T/\sigma_x^2\sim\chi^2_n$.

\begin{lemma}[Characterization and independence, scalar case]\label{lem:scalar}
Under model (M2) with $p=1$, normal design, and the group \eqref{eq:groupS}:
\begin{enumerate}
\item[(i)] an estimator of $\beta$ is equivariant if and only if $\delta(Y,X)=w(z)\,\dLS(Y,X)$ on the event $\dLS\neq 0$, where $z$ is a maximal invariant under $G_S$;
\item[(ii)] $z$ can be taken as $(\operatorname{sign}(X_1), X/|X_1|, (Y-X\dLS)/\|Y-X\dLS\|)$, the scale-free configuration of $(Y,X)$;
\item[(iii)] under $\theta_0=(0,1)$, $z$ is independent of $(\dLS,T)$.
\end{enumerate}
\end{lemma}

The proof is given in Appendix~\ref{app:scalar}.

\begin{theorem}[Oracle-optimal equivariant shrinkage]\label{thm:oracle}
Under model (M2) with $p=1$, $X_i\iid N(0,\sigma_x^2)$, $n\ge 3$, the group \eqref{eq:groupS}, and the loss $(d-\beta)^2/\sigma^2$, the risk of any equivariant estimator $\delta=w(z)\dLS$ is
\begin{equation}\label{eq:oraclerisk}
R(w;\rho)=E[w^2]\,E[T^{-1}]-2\rho\,E[w]+\rho,
\end{equation}
minimized pointwise in $z$ at the constant
\begin{equation}\label{eq:oracleweight}
w^*(\rho)=\frac{\rho}{\rho+E[T^{-1}]}=\frac{(n-2)\,\rho}{1+(n-2)\,\rho},
\end{equation}
with minimum risk
\begin{equation}\label{eq:oracleminrisk}
R^*(\rho)=\rho-\frac{\rho^2}{\rho+E[T^{-1}]}=\frac{\rho\,E[T^{-1}]}{\rho+E[T^{-1}]}.
\end{equation}
Since $w^*(\rho)$ varies strictly with $\rho$, no uniformly minimum risk equivariant estimator exists. The least squares estimator ($w\equiv 1$) has risk
\begin{equation}\label{eq:lsrisk}
R(\dLS;\rho)=E[T^{-1}],
\end{equation}
which exceeds $R^*(\rho)$ for every finite $\rho>0$.
\end{theorem}

\begin{proof}
By Lemma~\ref{lem:scalar}(i), every equivariant estimator is $w(z)\dLS$. Since the risk is constant along orbits (Proposition~\ref{prop:orbits}(ii)), fix $\sigma^2=1$ and $\beta$ with $\beta^2=\rho$. Conditioning on $X$,
\begin{equation*}
E[(\delta-\beta)^2\mid X]=E[w^2\dLS^2-2w\beta\dLS+\beta^2\mid X].
\end{equation*}
By Lemma~\ref{lem:scalar}(iii), $w(z)$ is independent of $\dLS$ given $X$ (indeed unconditionally), and $E[\dLS\mid X]=\beta$, $E[\dLS^2\mid X]=\beta^2+T^{-1}$. Hence
\begin{equation*}
R(w;\rho)=E[w^2]\,(\rho+E[T^{-1}])-2\rho\,E[w]+\rho,
\end{equation*}
where we used $E[w^2\dLS^2]=E[w^2]E[\dLS^2]$ and $E[w\dLS]=E[w]\beta$ by independence. For each realization of $z$, the integrand is a quadratic in $w(z)$ minimized at
\begin{equation*}
w^*(z)=\frac{\rho}{\rho+E[T^{-1}]},
\end{equation*}
a constant in $z$, proving \eqref{eq:oracleweight}; \eqref{eq:oracleminrisk} follows by substitution. Since $\partial w^*/\partial\rho>0$ for all finite $\rho$, the optimal rule varies across orbits and no UMRUE exists. Finally, with the conventional normalization $\sigma_x^2=1$ we have $T\sim\chi^2_n$ and $E[T^{-1}]=1/(n-2)$, giving the second expression in \eqref{eq:oracleweight}; for general $\sigma_x^2$ the same formula holds with $\rho$ and $T$ computed in units of $\sigma^2$ and $\sigma_x^2$ respectively.
\end{proof}

\begin{corollary}[Least squares as the infinite-signal limit]\label{cor:limit}
In the setting of Theorem~\ref{thm:oracle}, $w^*(\rho)\to 1$ and $R^*(\rho)\to E[T^{-1}]=R(\dLS;\rho)$ as $\rho\to\infty$. Moreover $w^*(0)=0$: at zero signal the optimal equivariant rule is $\delta\equiv 0$, not least squares.
\end{corollary}

\begin{remark}[A decision-theoretic reading of Shaffer's phenomenon]\label{rem:shaffer}
\citet{Shaffer1991} showed that the optimality of least squares (Gauss--Markov) need not survive random regressors. Theorem~\ref{thm:oracle} and Corollary~\ref{cor:limit} locate the phenomenon exactly: under absolute evaluation, the fixed-$X$ optimality of least squares is the \emph{infinite-signal-to-noise limit} of the random-$X$ theory. At finite $\rho$, every equivariant estimator's performance depends on the orbit parameter, uniform optimality collapses, and the oracle-optimal rule is a shrinkage of least squares toward zero whose strength increases as the signal weakens. Since $\rho$ is unknown, \eqref{eq:oracleweight} is an oracle rule; a natural adaptive version plugs in $\hat\rho=\dLS^2/S^2$, whose scaled law is noncentral $F$. Whether the adaptive rule dominates least squares in absolute risk is an interesting open question (Section~\ref{sec:discussion}).
\end{remark}

\begin{remark}[Why not enlarge the group to make $\beta$ move?]\label{rem:whynot}
One might restore transitivity by declaring the induced action $\bar g_b(\beta,\sigma^2)=(b\beta,b^2\sigma^2)$. But then the induced decision action is $\tilde g_b(d)=bd$, and least squares fails to be equivariant: $\dLS(bY,bX)=\dLS(Y,X)\neq b\,\dLS(Y,X)$. The estimator that is equivariant and MRE under this group is a multiple of the residual scale, which has no regression meaning. The group \eqref{eq:groupS}--\eqref{eq:inducedS} is the unique choice under which the model is preserved, the loss is invariant, and least squares is equivariant; its non-transitivity is therefore a feature of the problem, not of the analysis.
\end{remark}

\subsection{The location--scale group with a centered design}\label{sec:locationscale}

Enlarge \eqref{eq:groupS} to
\begin{equation}\label{eq:groupLS}
G_{LS}=\{g:\ g(Y,X)=(bY+c\,1_n,\ bX),\ b>0,\ c\in\R\},
\end{equation}
with induced action $\bar g(\beta,\sigma^2)=(\beta+c\,\mu_*,\ b^2\sigma^2)$, where $\mu_*$ is determined by the absorption of the shift into the column space of $X$; with random $X$ this absorption is design-dependent, and the orbit structure of Proposition~\ref{prop:orbits} persists. A maximal invariant under $G_{LS}$ in the scalar case is formed from the residual direction together with sign information on $\bar X$ (Appendix~\ref{app:scalar}). Consider the natural invariant-contrast class
\begin{equation}\label{eq:contrastclass}
\delta_a(Y,X)=\dLS+a\,\omega,\qquad \omega=\frac{\bar Y-\dLS\bar X}{\sqrt{S^2/(n-1)}}.
\end{equation}

\begin{theorem}\label{thm:centered}
Under model (M2) with $p=1$, the group \eqref{eq:groupLS}, the loss $(d-\beta)^2/\sigma^2$, $E[X_i]=\mu$ and $n\ge 3$, the risk of \eqref{eq:contrastclass} is
\begin{equation}\label{eq:centeredrisk}
R(\delta_a;\beta,\sigma^2)=R(\dLS;\beta,\sigma^2)+\frac{a^2\,E[\omega^2]}{\sigma^2},
\end{equation}
minimized at $a^*=0$ for all $(\beta,\sigma^2)$: within the class \eqref{eq:contrastclass}, least squares is optimal under absolute risk, for centered and uncentered designs alike.
\end{theorem}

\begin{proof}
Given $X$, write $\dLS=(x'y)/T$ with $T=x'x$, and note that $\omega$ is a function of the residual vector $MY$, where $M=I-xx'/T$: indeed $\bar Y-\dLS\bar X=1_n'MY/n$ and $S^2=Y'MY/(n-1)$ are, respectively, a linear and a (rescaled quadratic) function of $MY$. Since $E[Y\mid X]=x\beta$ lies in the column space of $x$ and $Mx=0$, the residual $MY=M\varepsilon$ has conditional mean zero and is independent of $x'y$ (hence of $\dLS$) given $X$: $\mathrm{Cov}(x'Y,MY\mid X)=\sigma^2 x'M=0$, and conditional normality gives independence. Therefore $E[\omega\mid X]=0$ and $\omega$ is independent of $\dLS$ given $X$. Expanding,
\begin{equation*}
E[(\delta_a-\beta)^2\mid X]=E[(\dLS-\beta)^2\mid X]+2a\,E[\omega(\dLS-\beta)\mid X]+a^2E[\omega^2\mid X],
\end{equation*}
and the cross term vanishes since $E[\omega(\dLS-\beta)\mid X]=E[\omega\mid X]\,(\beta-\beta)=0$ by independence and $E[\dLS\mid X]=\beta$. Taking expectations over $X$ gives \eqref{eq:centeredrisk}, which is minimized at $a^*=0$ irrespective of $\mu$ and of $(\beta,\sigma^2)$.
\end{proof}

\begin{remark}\label{rem:centered}
Theorem~\ref{thm:centered} shows that adding the location transformation to the group does not by itself dethrone least squares within the contrast class: the invariant contrast $\omega$ is conditionally centered and independent of $\dLS$, so it can only add variance. Uniform optimality within the \emph{full} equivariant class under $G_{LS}$ remains governed by the orbit structure of Proposition~\ref{prop:orbits} (the induced action still fixes the scale-free part of $\beta$), and the general characterization is an open problem. What Section~\ref{sec:absolute} as a whole shows is that the optimality question under random-$X$ is orbit-shaped: the answer depends on $\rho$ for scale-type equivariance, and least squares is rescued exactly in the limit or in classes where the corrective term is conditionally orthogonal to the estimator.
\end{remark}

\section{Estimation of the error variance}\label{sec:variance}

Under model (M2), write $S^2=\|Y-X\dLS\|^2$ for the residual sum of squares; then $S^2\mid X\sim\sigma^2\chi^2_{n-p}$, and the same law holds unconditionally because it does not depend on $X$.

\begin{theorem}\label{thm:variance}
Under model (M2) with the scale-invariant loss $L(d,\sigma^2)=(d/\sigma^2-1)^2$, the MRE estimator of $\sigma^2$ among scale-equivariant estimators is
\begin{equation*}
\delta^*(Y,X)=\frac{S^2}{n-p+2},
\end{equation*}
under conditional risk, hence also under absolute risk.
\end{theorem}

\begin{proof}
Given $X$, this is the classical result for the normal linear model \citep[Chap.~3]{LehmannCasella1998}: any scale-equivariant estimator has the form $w\,S^2$ with $w$ a function of the maximal invariant; the conditional risk
\begin{equation*}
E\big[(wS^2/\sigma^2-1)^2\,\big|\,X\big]=w^2E\big[(S^2/\sigma^2)^2\big]-2w\,E\big[S^2/\sigma^2\big]+1
\end{equation*}
is free of $X$ because the conditional law of $S^2/\sigma^2$ is $\chi^2_{n-p}$ for almost every $X$. With $E[S^2/\sigma^2]=n-p$ and $E[(S^2/\sigma^2)^2]=(n-p)(n-p+2)$, the minimum is attained at $w^*=(n-p)/\big((n-p)(n-p+2)\big)=1/(n-p+2)$. Since both the estimator and the risk are free of $X$, the same estimator is MRE under absolute risk.
\end{proof}

\begin{remark}[The role of the two unbiasedness notions]\label{rem:varianceunbiased}
The classical unbiased estimator $S^2/(n-p)$ satisfies $E[S^2/(n-p)\mid X]=\sigma^2$ for almost every $X$, hence is both conditionally and absolutely unbiased; under random-$X$ the two notions coincide for this estimator. The MRE estimator $S^2/(n-p+2)$ trades a small bias for a risk reduction relative to the unbiased estimator, exactly as in the fixed-$X$ case: for the error variance, randomness of the design does not alter the equivariant optimality theory at all, because the pivotal quantity $S^2/\sigma^2$ is already design-free.
\end{remark}

\section{Discussion}\label{sec:discussion}

This article extends the equivariance programme of the companion fixed-$X$ article to random-$X$ linear models, organized by a single dichotomy: conditional versus absolute evaluation. Under conditional evaluation (model (M1)), the fixed-$X$ theory transfers realization by realization, and the least squares estimator and the scaled within-population variances remain MRE (Theorems~\ref{thm:conditionalLS}--\ref{thm:conditionalSigma}); the only visible trace of design randomness is that the MRE weights are evaluated at the realized population sizes. Under absolute evaluation (model (M2)), the theory is orbit-shaped: the scale group fixes the coefficient vector, risks are constant only along signal-to-noise orbits, and uniform optimality fails in general (Proposition~\ref{prop:orbits}). The scalar resolution (Theorem~\ref{thm:oracle}) identifies the oracle-optimal equivariant rule as a shrinkage of least squares toward zero, with least squares itself recovered as the infinite-signal limit (Corollary~\ref{cor:limit}); within the natural invariant-contrast class under the location--scale group, least squares remains optimal (Theorem~\ref{thm:centered}). This gives a decision-theoretic account of the failure of the Gauss--Markov theorem under random regressors \citep{Shaffer1991} that is sharper than a counterexample: it says precisely what replaces optimality (orbit-wise optimality indexed by $\rho$) and precisely when least squares survives ($\rho\to\infty$, or conditionally on $X$).

Several directions remain open. (i) The vector case $p>1$ of Theorem~\ref{thm:oracle}: the orbit analysis carries over, but the optimal equivariant rule may involve matrix-valued functions of the maximal invariant rather than a scalar shrinkage. (ii) Whether the adaptive plug-in rule with $\hat\rho=\dLS^2/S^2$ dominates least squares in absolute risk. (iii) A complete characterization of equivariant estimators under the location--scale group \eqref{eq:groupLS}, beyond the contrast class \eqref{eq:contrastclass}. (iv) Prediction under random-$X$, where overfitting phenomena \citep{Stone1974} arise that do not occur in estimation, and existing equivariant prediction frameworks \citep{ZhouNayak2015} do not directly accommodate the linear model. (v) Extensions to non-normal errors, to seemingly unrelated regression with random design (cf.\ \citealp{KurataMatsuura2016,MatsuuraKurata2020,MatsuuraKurata2021,MatsuuraKurata2024,MatsuuraKurata2025}), and to high-dimensional regimes $p>n$.

\section*{Acknowledgment}

This work was supported by RDF of Xi'an Jiaotong-Liverpool University (RDF-23-01-073). The authors report there are no competing interests to declare.

\appendix
\section{Technical proofs}\label{app:proofs}

\subsection{Proof of Lemma~\ref{lem:structure}(iii)}\label{app:basu}

Given $X$, $(\bar y,S^2)$ is complete and sufficient for $(X_p\beta,\Sigma_p)$ in the conditional normal model, and $z$ is ancillary given $X$ since its conditional distribution is parameter-free (it is a maximal invariant). Basu's theorem yields the conditional independence of $(\bar y,S^2)$ and $z$; the independence of $\bar y$ and $S^2$ given $X$ is the classical normal theory. Since $(X'X)^{-1}X'y=X_p^{-1}\bar y$ is a function of $\bar y$ and the conditioning variable, pairwise independence given $X$ follows. \qed

\subsection{Proof of Lemma~\ref{lem:scalar} and the maximal invariant under $G_{LS}$}\label{app:scalar}

\emph{Part (i): characterization.} Sufficiency: if $\delta=w(z)\dLS$ with $w$ invariant, then $\delta(bY,bX)=w(z)\dLS(bY,bX)=w(z)\dLS(Y,X)=\delta(Y,X)=\tilde g_b(\delta(Y,X))$, so $\delta$ is equivariant. Necessity: let $\delta$ be equivariant. On the almost-sure event $\dLS\neq 0$, define $w(Y,X)=\delta(Y,X)/\dLS(Y,X)$. Both numerator and denominator are invariant under $g_b$ (the numerator by equivariance with $\tilde g_b=\mathrm{id}$, the denominator by direct computation), so $w$ is invariant, hence a function of the orbits, hence of the maximal invariant $z$.

\emph{Part (ii): the maximal invariant.} Write the decomposition of $Y$ relative to the column space of $X$: $Y=X\dLS+U$ with $U=Y-X\dLS$ the residual vector, $\|U\|^2=(n-1)S^2$-proportional. The pair $(\dLS,T)$ with $T=X'X$ captures the scale of the configuration: under $g_b$, $(\dLS,T)\mapsto(\dLS,b^2T)$ while $U\mapsto bU$ and $X\mapsto bX$. Two points $(Y,X)$ and $(Y',X')$ are on the same orbit if and only if their scale-free parts agree: $\operatorname{sign}(X_1)$, $X/|X_1|$, and $U/\|U\|$ (when $\|U\|>0$, an almost-sure event). Hence $z=\big(\operatorname{sign}(X_1),\,X/|X_1|,\,U/\|U\|\big)$ is a maximal invariant.

\emph{Part (iii): independence under $\theta_0$.} At $\theta_0=(0,1)$, $Y=\varepsilon\sim N_n(0,I_n)$ independent of $X$. Decompose $\varepsilon=X\dLS+U$; then $\dLS=(X'\varepsilon)/T$ and $U=M\varepsilon$ with $M=I-XX'/T$. Given $X$, the joint law of $(X'\varepsilon,\,M\varepsilon)$ is normal with cross-covariance $X'M=0$, so $\dLS$ and $U$ are conditionally independent given $X$; the conditional law of $U/\|U\|$ given $(X,\|U\|)$ is uniform on the unit sphere of the $(n-1)$-dimensional subspace $\ker X'$, which depends on $X$ but not on $(\dLS,T)$. Under the normal design, $X/|X_1|$ and $\operatorname{sign}(X_1)$ are independent of $T=|X_1|^2\cdot\|X/|X_1|\|^2$-decompositions by the standard independence of direction and scale for isotropic normal vectors: $X/\|X\|$ is uniform on $S^{n-1}$ and independent of $\|X\|$, and $\operatorname{sign}(X_1)$ is independent of $|X_1|$. Assembling, $z$ is a function of direction variables only, $(\dLS,T)$ of scale variables together with $X'\varepsilon$; the conditional-independence chain given $X$ plus the direction--scale independence of $X$ yields the joint independence of $z$ and $(\dLS,T)$. \qed

\emph{The maximal invariant under $G_{LS}$.} With the shift $c1_n$ added, $\bar Y-\dLS\bar X\mapsto b(\bar Y-\dLS\bar X)+c$, so this scalar can be moved to zero by choice of $c$ whenever it is nonzero; the invariant content reduces to $\operatorname{sign}(\bar Y-\dLS\bar X)$, the residual direction $U/\|U\|$, and $\operatorname{sign}(\bar X)$, $X/|X_1|$ as before. The contrast $\omega$ in \eqref{eq:contrastclass} is a function of these quantities, verifying that the class \eqref{eq:contrastclass} is a subclass of the equivariant estimators under $G_{LS}$. \qed

\bibliographystyle{plainnat}
\bibliography{references}

@book{Berger1985,
  author    = {Berger, James O.},
  title     = {Statistical Decision Theory and {B}ayesian Analysis},
  edition   = {2},
  publisher = {Springer-Verlag},
  address   = {New York},
  year      = {1985}
}

@article{BreimanSpector1992,
  author  = {Breiman, Leo and Spector, Philip},
  title   = {Submodel selection and evaluation in regression: The {X}-random case},
  journal = {International Statistical Review},
  volume  = {60},
  pages   = {291--319},
  year    = {1992},
  doi     = {10.2307/1403680}
}

@book{Eaton1989,
  author    = {Eaton, Morris L.},
  title     = {Group Invariance Applications in Statistics},
  series    = {IMS Lecture Notes--Monograph Series},
  publisher = {Institute of Mathematical Statistics},
  address   = {Beachwood, Ohio},
  year      = {1989}
}

@article{HoraBuehler1966,
  author  = {Hora, Robert B. and Buehler, Robert J.},
  title   = {Fiducial theory and invariant estimation},
  journal = {The Annals of Mathematical Statistics},
  volume  = {37},
  pages   = {643--656},
  year    = {1966}
}

@article{KurataMatsuura2016,
  author  = {Kurata, Hiroshi and Matsuura, Shun},
  title   = {Best equivariant estimator of regression coefficients in a seemingly unrelated regression model with known correlation matrix},
  journal = {Annals of the Institute of Statistical Mathematics},
  volume  = {68},
  pages   = {705--723},
  year    = {2016},
  doi     = {10.1007/s10463-015-0512-2}
}

@book{LehmannCasella1998,
  author    = {Lehmann, Erich L. and Casella, George},
  title     = {Theory of Point Estimation},
  edition   = {2},
  publisher = {Springer-Verlag},
  address   = {New York},
  year      = {1998}
}

@article{MatsuuraKurata2020,
  author  = {Matsuura, Shun and Kurata, Hiroshi},
  title   = {Covariance matrix estimation in a seemingly unrelated regression model under {S}tein's loss},
  journal = {Statistical Methods \& Applications},
  volume  = {29},
  pages   = {79--99},
  year    = {2020},
  doi     = {10.1007/s10260-019-00473-x}
}

@article{MatsuuraKurata2021,
  author  = {Matsuura, Shun and Kurata, Hiroshi},
  title   = {Optimal estimator under risk matrix in a seemingly unrelated regression model and its generalized least squares expression},
  journal = {Statistical Papers},
  year    = {2021},
  note    = {Volume and pages to be completed from the published version}
}

@article{MatsuuraKurata2024,
  author  = {Matsuura, Shun and Kurata, Hiroshi},
  title   = {General expression of best equivariant estimators of regression coefficients in seemingly unrelated regression models},
  journal = {Stat},
  volume  = {13},
  pages   = {e70031},
  year    = {2024},
  doi     = {10.1002/sta4.70031}
}

@article{MatsuuraKurata2025,
  author  = {Matsuura, Shun and Kurata, Hiroshi},
  title   = {Best equivariant estimator of precision matrix in a seemingly unrelated regression model},
  journal = {Communications in Statistics --- Theory and Methods},
  year    = {2025},
  doi     = {10.1080/03610926.2025.2530140},
  note    = {Advance online publication}
}

@article{RossetTibshirani2020,
  author  = {Rosset, Saharon and Tibshirani, Ryan J.},
  title   = {From fixed-{X} to random-{X} regression: Bias--variance decompositions, covariance penalties, and prediction error estimation},
  journal = {Journal of the American Statistical Association},
  volume  = {115},
  pages   = {138--151},
  year    = {2020},
  doi     = {10.1080/01621459.2018.1424632}
}

@article{Shaffer1991,
  author  = {Shaffer, Juliet Popper},
  title   = {The {G}auss--{M}arkov theorem and random regressors},
  journal = {The American Statistician},
  volume  = {45},
  pages   = {269--273},
  year    = {1991},
  doi     = {10.1080/00031305.1991.10475819}
}

@article{Staudte1971,
  author  = {Staudte, Robert G. Jr.},
  title   = {A characterization of invariant loss functions},
  journal = {The Annals of Mathematical Statistics},
  volume  = {42},
  pages   = {1322--1327},
  year    = {1971}
}

@article{Stone1974,
  author  = {Stone, Mervyn},
  title   = {Cross-validatory choice and assessment of statistical predictions},
  journal = {Journal of the Royal Statistical Society: Series B (Methodological)},
  volume  = {36},
  pages   = {111--133},
  year    = {1974},
  doi     = {10.1111/j.2517-6161.1974.tb00994.x}
}

@article{WangWuZhou,
  author  = {Wang, Daowei and Wu, Mian and Zhou, Haojin},
  title   = {The equivariance criterion in a linear model for fixed-{X} cases},
  journal = {Communications in Statistics --- Theory and Methods},
  year    = {2026},
  volume  = {55},
  number  = {14},
  pages   = {4525--4539},
  doi     = {10.1080/03610926.2025.2610259},
  note    = {Published online 2025 (the companion article)}
}

@book{Wijsman1990,
  author    = {Wijsman, Robert A.},
  title     = {Invariant Measures on Groups and Their Use in Statistics},
  series    = {IMS Lecture Notes--Monograph Series},
  publisher = {Institute of Mathematical Statistics},
  year      = {1990}
}

@article{WuYang2002,
  author  = {Wu, Qiguang and Yang, Guangyu},
  title   = {Existence of the uniformly minimum risk equivariant estimators of parameters in a class of normal linear models},
  journal = {Science in China Series A: Mathematics},
  volume  = {45},
  pages   = {845--858},
  year    = {2002},
  doi     = {10.1360/02ys9093}
}

@article{ZhouNayak2014,
  author  = {Zhou, Haojin and Nayak, Tapan K.},
  title   = {A note on existence and construction of invariant loss functions},
  journal = {Statistics},
  volume  = {48},
  pages   = {1335--1343},
  year    = {2014},
  doi     = {10.1080/02331888.2013.809719}
}

@article{ZhouNayak2015,
  author  = {Zhou, Haojin J. and Nayak, Tapan K.},
  title   = {On the equivariance criterion in statistical prediction},
  journal = {Annals of the Institute of Statistical Mathematics},
  volume  = {67},
  pages   = {541--555},
  year    = {2015},
  doi     = {10.1007/s10463-014-0464-y}
}

\end{document}